\documentclass[12pt]{article}

\usepackage[T1]{fontenc}
\usepackage{amsmath}
\usepackage{amssymb}
\usepackage{amsthm}
\usepackage{mathtools}
\usepackage{url}
\usepackage[T1]{fontenc}
\usepackage{lmodern}
\usepackage[utf8]{inputenc}              
\usepackage{bigfoot} 
\usepackage{filecontents}

\usepackage{pgfplots}	
\pgfplotsset{compat=newest} 
\usepackage{geometry}
\usepackage{leftidx}

\usepackage{hyperref}
\hypersetup{
	pdftoolbar=true,
	pdfmenubar=true,
	pdffitwindow=false,
	pdfstartview={FitH},
	pdftitle={},
	pdfauthor={Dario Corona},
	pdfsubject={},
	pdfkeywords={},
	pdfnewwindow=true,
	colorlinks=true, 
	linkcolor=blue,
	citecolor=blue,
	urlcolor=blue,
}

\DeclarePairedDelimiter{\abs}{|}{|}

\theoremstyle{plain}\newtheorem{teo}{Theorem}
\theoremstyle{plain}\newtheorem{proposition}[teo]{Proposition}
\theoremstyle{plain}\newtheorem{lemma}[teo]{Lemma}
\theoremstyle{plain}\newtheorem{corollary}[teo]{Corollary}
\theoremstyle{plain}\newtheorem{conjecture}[teo]{Conjecture}

\theoremstyle{definition}\newtheorem{definition}[teo]{Definition}
\theoremstyle{remark}\newtheorem{remark}[teo]{Remark}
\theoremstyle{definition}

\numberwithin{teo}{section} 
\numberwithin{equation}{section} 

\usepackage{color}

\title{Dynamical properties of critical exponent functions
}

\author{Dario Corona, Alessandro Della Corte, Marco Farotti
	\\
	\small{School of Science and Technology,
University of Camerino (Italy)}}
\date{}

\begin{document}
\maketitle

\begin{abstract}

In the last years the attention towards topological dynamical properties of highly discontinuous maps has increased significantly. In \cite{Critical}, a class of densely discontinuous interval maps, called ``critical exponent maps'', was introduced. These maps are defined through the word-combinatorics concept of critical exponent applied to the binary expansion of reals and show highly chaotically properties as well as some challenging problems. In this paper we identify an error in the proof of Theorem 7 in \cite{Critical}, a purely combinatorial result which in fact does not hold. We show that most of the results in \cite{Critical}, obtained there through Theorem 7, can be recovered. Moreover, we propose as a conjecture a weaker form of Theorem 7.
	


\end{abstract}
\section{Introduction}

In recent decades, emphasis has been placed on the concept of
\textit{critical exponent} of a word $w$ over a given alphabet, that is the
supremum of the reals $\alpha$ for which $w$ contains an $\alpha$-power.
The critical exponents are significant in fields such that symbolic dynamics and transcendental number theory \cite{adamczewski2007dynamics,adamczewski2010expansion}.
In particular, this concept is closely linked to the problem of repetition threshold of an alphabet see \cite[p. 126]{berstel2008combinatorics}), which was solved quite recently \cite{currie2011pr,queffelec2006old,Rao2011}.

In \cite{Critical}, a family of interval maps based
on the word-combinatorial concept of critical exponent was introduced. 
The generic $n$-\textit{critical exponent function}
$\kappa_n:[0,1]\to[0,1]$ is the map that associates to every $x\in [0,1]$, the inverse of the critical exponent of the (infinite) base-$n$ expansion of $x$
(we assume $1/\infty=0$).  A word-combinatorial result,
namely~\cite[Theorem 7]{Critical}, was presented
and employed in that paper to prove
some analytical and dynamical properties of 
such interval maps.
However,~\cite[Theorem 7]{Critical} does not hold.
Denoting $\{0,1\}^*$ and $\{0,1\}^\omega$ as the sets of
finite and infinite binary words, respectively,
and using $E(w)$ to represent the critical exponent of a word
(see Definition~\ref{def:crit-exp}),
Theorem 7 in \cite{Critical} was formulated as follows:
\begin{center}
``\textit{Let $w \in \left\{0,1\right\}^*$.
For every $\alpha > \max\left\{2,E(w)\right\}$,\\
there exists $y \in \left\{0,1\right\}^\omega$ such that
$E(wy) = \alpha$}''.
\end{center}
This claim is false, as shown for instance by the counterexample $w=00100100$.
Indeed, in this case we have $E(w)=8/3$,
while both $w0$ and $w1$ are 3-powers,
so that $E(wy)\ge 3 > 8/3$ for every $y\in\{0,1\}^\omega$.
We thank Sa\'ul Rodr\'iguez Mart\'in for 
this remark. 

We think that 
a result of similar flavor
is true
(see Conjecture~\ref{conj:ourCon}). Even if in the present work we managed to walk around the use of \cite[Theorem 7]{Critical} to establish the main properties of the critical exponent maps, we believe that the conjecture is interesting in itself. 

The paper is organized as follows. In Section~\ref{sec:preliminaries}, we introduce
the preliminary concepts and the notation that will be used throughout 
the paper.
Section~\ref{sec:combinatorial} presents
our conjecture that would replace~\cite[Theorem 7]{Critical},
together with some remarks and weaker results 
useful in the subsequent analysis.
Section~\ref{sec:properties} introduces the critical exponent functions
and examines their properties,
taking into account the error identified in~\cite[Theorem 7]{Critical}.
Finally, we highlight the claims made in \cite{Critical} which are still not
recovered.

\section{Setting and notation}\label{sec:preliminaries}
We denote the set of all natural numbers (i.e.  positive integers) by $\mathbb{N}$ and we set $\mathbb{N}_0=\mathbb{N}\cup \{0\}$. 
Let $\mathcal{A} = \left\{a_0,a_1,\dots, a_{n}\right\}$ be an alphabet. We indicate
by $\mathcal{A}^{*}$ and $\mathcal{A}^\omega$
the set of the finite and infinite words over $\mathcal{A}$, respectively.  We set 
$\mathcal{A}^\infty = \mathcal{A}^* \cup \mathcal{A}^{\omega}$.
We let $\epsilon$ denote the empty word
and we set $\mathcal{A}^+ = \mathcal{A}^*\setminus\{\epsilon\}$.
Given a word $w \in \mathcal{A}^{\infty}$,
we indicate by $w_i$ its $i$-th digit.
Let $\ell\colon \mathcal{A}^* \to \mathbb{N}$ be the map that associates
to the nonempty finite word $w=w_1\dots w_n$ the natural number $n$,
while we set $\ell(\epsilon)=0$.
For $v, w \in \mathcal{A}^\infty$, we say that $v$ is a \textit{subword} of $w$ if there exist
$p \in \mathcal{A}^*$ and
$s \in \mathcal{A}^\infty$ 
such that $w = p v s$.
We say that $v \in \mathcal{A}^+$
is a \textit{prefix} (\textit{suffix}) of $w$
if $w = v s$ ($w = p v$).
For $w \in \mathcal{A}^{\omega}$,
we indicate by $\mathcal{L}(w)$ the \textit{language} of $w$,
that is the set  of all the finite subwords of $w$. 

For every $w \in \mathcal{A}^*$ and $n \in \mathbb{N}$,
we indicate by $w^n$ the concatenation of $n$-copies of $w$,
that is
\[
	w^n = \underbrace{w w \dots w}_{\text{$n$-times}}.
\]
We call \textit{deletion operator} the function $\delta\colon \mathcal{A}^\infty \to \mathcal{A}^\infty$ that removes the first digit of a word.
Hence, for a word $w = w_0 w_1 w_2   \ldots \in \mathcal{A}^{\infty}$, we have that
\[
	\delta(w) = \delta(w_0 w_1 w_2  \dots) = w_1 w_2 \dots \ 
\]
and we set $\delta(\epsilon) = \epsilon$.
Consequently,  we can define for every $n \in \mathbb{N}$,
the operator $\delta^n$ that removes the first $n$ letters of a word. Thus,  if $w = w_0 w_1 w_2   \ldots \in \mathcal{A}^{\infty}$, we have that
\[
	\delta^n(w)= 
	\delta^n(w_0 w_1 \dots\ w_n w_{n+1} \dots )
	= w_n w_{n+1} \dots \ .
\]
For every $x \in [0,1]$,  we denote by $w_x \in \left\{0,1\right\}^\omega$ the binary expansion of $x$, that is
\[
	x = (0.w_x)_2 = \sum_{i = 0}^{\infty}\frac{(w_x)_{i}}{2^i}.
\]
\begin{remark}
	If $x \in [0,1]$ is a dyadic rational,
	we consider its binary expansion $w_x$, whose digits are ultimately $1$,
	for instance $w_{\frac{1}{2}} = 01111\dots$.
\end{remark}
For any finite binary word $w\in\{0,1\}^+$,
we denote by $I_w$ the open cylinder characterized
by $w$, hence
$ I_w \coloneqq \left((0.w)_2,\, (0.w)_2 + 2^{-\ell(w)}\right) $.

\bigskip
For positive integers $p$ and $q$,
we say that a word $w \in \mathcal{A}^+$ is a \textit{$p/q$-power} 
if it is of the form
$w = x^n y$, where $x \in \mathcal{A}^+$, $y\in\mathcal{A}^*$ is a prefix of $x$, $\ell(w) = p$ and $\ell(x) = q$ (cf. \cite{Dejean1972}).
Let $\alpha \in \mathbb{Q}^+$.
A word
$w\in \mathcal{A}^\infty\setminus\{\epsilon\}$
\emph{avoids $\alpha$-powers}
if none of its subwords is an $r$-power
for any rational $r\ge \alpha$.
Otherwise, we say that $w$ \textit{contains an $\alpha$-power}.
\begin{definition}
	\label{def:crit-exp}
	The \textit{critical exponent} of a word 
	is the function $E \colon \mathcal{A}^\infty \to \mathbb{R}$, given by 
	\begin{equation*}
		E(w) =
		\begin{cases}
			\sup\left\{r \in \mathbb{Q}:
				w \text{ contains an $r$-power}	 
			\right\}, & \mbox{if } w \ne \epsilon,\\
			0, &		\mbox{if } w = \epsilon. 
		\end{cases}
	\end{equation*}
\end{definition}


Let us define the well known Thue-Morse sequence, which is important for fields such as number theory, differential geometry,  combinatorics of words,  and Morse theory (\cite{allouche1999ubiquitous}). 

\begin{definition}
	\label{def:ThueMorse}
	The \textit{Thue-Morse sequence} is the element
	$\tau = \tau_0 \tau_1 \tau_2 \dots \tau_n
	\ldots \in \left\{0,1\right\}^\omega$
	given by 
	\[
		\begin{cases}
			\tau_0 = 0, \\
			\tau_{2n} = \tau_n, \\
			\tau_{2n +1 } = 1 - \tau_n.
		\end{cases}
	\]
	The Thue-Morse sequence starts as follows:
	\[
		\tau = 0\ 1\ 10\ 1001\ 10010110\
		1001011001101001\ \dots\  .
	\]
	It has been proved that $ E(\tau) = 2$ (see the classical paper  \cite{thue1912}).
\end{definition}

\begin{remark}
	\label{rem:ThueMorse-bitwiseConstruction}
	By recursively using bitwise negation, an alternative
	construction can be employed to define the Thue-Morse sequence.
	The first digit of the sequence is $\tau_0 = 0$.
	Then, for any $n\in\mathbb{N}_0$,  if the first $2^n$ digits are given,
	then the next $2^n$ digits of the sequence are
	the bitwise negations of the formers.
	In other words, if we have
	$\tau_0,\tau_1,\dots,\tau_{2^n - 1}$,
	then
	$\tau_{k+2^n} = 1 - \tau_{k}$
	for every $k = 0,\dots,2^{n}-1$.
\end{remark}

\section{A conjecture and some combinatorial results}
\label{sec:combinatorial}

In this section, we present our conjecture replacing the ~\cite[Theorem 7]{Critical} along with some remarks and weaker results. The latter results, expressed in Proposition \ref{prop:w-subword-TM} and \ref{prop:weakTheorem},  will be useful for the analysis of dynamical and analytical properties of critical exponent functions done in the next section.

In order to present the conjecture, we give the following notation.
\begin{definition}
For any $w \in \{0,1\}^*$,
we denote by $\mathcal{P}(w)\subset \{0,1\}^{\omega}$
the set of binary words of infinite length with prefix $w$,
hence 
\begin{equation}
	\label{eq:defPw}
	\mathcal{P}(w) \coloneqq \left\{
		wy: y \in \{0,1\}^{\omega}
	\right\},
\end{equation}
and we denote by $\mathcal{E}_w$
the infimum among all the critical exponent 
of the words in $\mathcal{P}(w)$, namely
\begin{equation}
	\label{eq:def-calEw}
	\mathcal{E}_w \coloneqq
	\inf_{z \in \mathcal{P}(w)} E(z).
\end{equation}
\end{definition}

\begin{conjecture}
	\label{conj:ourCon}
	For any $w \in \{0,1\}^*$
	and $\alpha \ge \mathcal{E}_w$,
	there exists $y \in \{0,1\}^\omega$
	such that
	$E(wy) = \alpha$.
	In other words,
	\[
		E(\mathcal{P}(w)) =
		\left\{ E(z): z \in \mathcal{P}(w) \right\}
		=
		[\mathcal{E}_{w},+\infty].
	\]
\end{conjecture}

Even proving that for any $w\in \{0,1\}^*$
the infimum $\mathcal{E}_w$ is attained
does not appear straightforward.
However, this is not the primary focus of the result.
In particular, we find the most intriguing aspect to be
the proof that $E(\mathcal{P}(w))$ cannot be a disconnected set;
hence, it must be either
$E(\mathcal{P}(w)) = (\mathcal{E}_w,+\infty]$
or $E(\mathcal{P}(w)) = [\mathcal{E}_w,+\infty]$.

Notice that the previous conjecture holds true
if $w = \epsilon$,
as proved by 
J.~D. Currie and N.~Rampersad
in~\cite{currie2008each}.
Moreover, by using the construction 
of~\cite{currie2008each},
it holds also whenever $w$ is a finite
subword of the Thue-Morse sequence.
More formally, we have the following result, which in itself is an (easy) generalization of the result proved by Currie and Rampersad.
\begin{proposition}
	\label{prop:w-subword-TM}
	Let $w\in\{0,1\}^*$ be a finite subword of the 
	Thue-Morse sequence.
	Then, for any $\alpha \ge 2$
	there exists $y \in \{0,1\}^\omega$
	such that $E(wy) = \alpha$.
\end{proposition}

\begin{remark}
	\label{rem:currieConstruction}
	To provide a clearer and more readable proof of
	Proposition~\ref{prop:w-subword-TM}, let us briefly review
	the key points we need from the construction presented
	in~\cite{currie2008each} for a word
	$z \in \{0,1\}^{\omega}$ with critical exponent $\alpha$,
	for any $\alpha \ge 2$.
	The interest reader can find the details in the original paper.

	If $\alpha = 2$, then just consider the Thue-Morse sequence,
	namely $z = \tau$.

	Otherwise, let $(\beta_i)_{i \in \mathbb{N}}\subset \mathbb{R}$
	a sequence of real numbers converging to $\alpha$
	such that for any $i \in \mathbb{N}$
	we have
	$\beta_i < \alpha$ 
	and $\beta_i = r_i - t_i/2^{s_i}$
	for some positive integers $r_i,t_i,s_i$ with $s_i \ge 3$
	and $t_i < 2^{s_i}$.
	Moreover, denoting by $\mu$ the \emph{Thue-Morse morphism},
	i.e., $\mu(0) = 01$ and $\mu(1) = 10$,
	assume that $\delta^{t_i}\mu^{s_i}(0)$
	begins with $00$.
	By \cite[Lemma 5]{currie2008each},
	the word $\delta^{t_i}\mu^{s_i}(0^{r_i})$
	has critical exponent in $[\beta_i,\alpha)$,
	for any $i \in \mathbb{N}$.
	For any word $w$ (here some technical details are omitted),
	let us denote by
	$\phi_i(w)$ the word given by
	$\delta^{t_i}\mu^{s_i}(0^{r_i}w)$.
	Then, by construction and by \cite[Lemma 5]{currie2008each},
	the sequence of words $(w_n)_{n \in \mathbb{N}}$
	defined as 
	\[
		w_n = \phi_1(\phi_2(\dots (\phi_n(\epsilon)))),
	\]
	is such that $\ell(w_n)$ is diverging to infinity
	and $w_n$ is a prefix of $w_{n+1}$,
	so that the word $z = \lim_{n \to \infty}w_n$
	is well defined.
	Moreover, $E(z) \ge E(w_n)$ for every $n$
	and, at the same time, $E(p) < \alpha$ for every 
	prefix $p$ of $z$, implying that 
	$E(z) = \alpha$.

	By the previous construction, 
	we have that $z$ attains its critical exponent
	only ``at the infinity'', meaning that 
	$E(\delta^n(z)) = \alpha$ for any $n \in \mathbb{N}$.
	Moreover, $z$ always begins with $00$
	and it doesn't contain any $000$, since it has been obtained 
	through an iterative application of the Thue-Morse morphism.
	As a consequence, $z$ starts with $001$.
\end{remark}

\begin{proof}[Proof of Proposition~\ref{prop:w-subword-TM}]
	Since $w$ is a subword of $\tau$,
	there exists a prefix of the Thue-Morse sequence
	that contains $w$ as a subword.
	Then, the thesis can be obtained 
	by showing that for any $\alpha  \ge 2$ and
	for any finite prefix $p \in \{0,1\}^*$
	of the Thue-Morse sequence
	there exists a word $\tilde{y}$ that starts with $p$,
	$E(\tilde{y}) = \alpha$ and $E(\delta^n(\tilde{y})) = \alpha$
	for any $n \in \mathbb{N}$.
	Indeed, if the previous statement holds,
	the word $y$ can be selected by properly applying 
	$\delta$ to $\tilde{y}$.

	If $\alpha = 2$, this is trivial by considering 
	$\tilde{y} = \tau$.
	Otherwise, $\tilde{y}$ can be constructed 
	by choosing properly $\beta_1$
	in the construction given in Remark~\ref{rem:currieConstruction}.
	Indeed, let us choose $r_1 = 2$,
	$s_1 \ge 3$ such that $2^{s_1} > \ell(p)$
	and $t_1 = 5$,
	so that $\delta^{t_1}\mu^{s_1}(0^{r_1})$
	starts with $00$.
	Moreover, $s_1$ has to be chosen 
	so that $\beta_1 = r_1 - t_1/2^{s_1} < \alpha$.
	Denoting by $\tau_{(s_1)}$
	the prefix of $\tau$ with length $2^{s_1}$,
	we have that $p$ is a prefix of $\tau_{(s_1)}$.
	Moreover, we have that 
	$\mu^{s_1}(0^{r_1}) = \mu^{s_1}(00)
	= \tau_{(s_1)}\tau_{(s_1)}$,
	so that $\phi_1(\epsilon) = \delta^{t_1}\mu^{s_1}(00)
	=\delta^{5}(\tau_{(s_1)})\tau_{(s_1)}$.
	Therefore, choosing $\bar{n} = 2^{s_1} - t_1$,
	we have that 
	$\delta^{\bar{n}}\phi_1(\epsilon) = \tau_{(s_1)}$.
	Let $z$ be a word obtained through the construction 
	of Remark~\ref{rem:currieConstruction}
	with $\beta_1 = r_1 - t_1/2^{s_1}$.
	Then we have that 
	$\tilde{y} = \delta^{\bar{n}}(z)$
	starts with $p$ and,
	being a suffix of $z$, we also have
	that $E(\tilde{y}) = \alpha$
	and $E(\delta^n(\tilde{y})) = \alpha$ for any $n \in \mathbb{N}$,
	as required.
\end{proof}

 While the validity of Conjecture~\ref{conj:ourCon} remains uncertain,
we present the following weaker result,
which will prove useful in our subsequent analysis.

\begin{proposition}
	\label{prop:weakTheorem}
	For any $w\in \{0,1\}^*$,
	we have
	$[\ell(w),+\infty]\subset E(\mathcal{P}(w))$.
\end{proposition}

\begin{proof}
	The idea is
	to concatenate
	$w$ with a finite word $0^n$,
	where the value of $n$ will be fixed and determined later, 
	and then append to it an infinite word
	$y$ obtained using the construction in~\ref{rem:currieConstruction}.

	More formally, we proceed as follows.
	If $\ell(w) \le 2$, then the thesis
	holds as a direct consequence 
	of Proposition~\ref{prop:w-subword-TM}.
	Therefore, we assume $\ell(w) \ge 3$ and,
	without loss of generality,
	that the last digit of $w$ is $0$.
	Otherwise, we work with the bitwise negation of $w$
	and then we apply a bitwise negation 
	to the final construction once again.
	Hence, $w$ has the form $\tilde{w}0^k$,
	where $\tilde{w}$ is either the empty word
	or it ends with $1$.
	Let $n = \ell(w) - k $,
	and consider the word
	\[
		w0^n = 
		\tilde{w}0^k\, 0^n
		= \tilde{w}\underbrace{0\dots 0}_{\ell(w) }.
	\]

    For any $\alpha \ge \ell(w)$,
	consider $y \in \{0,1\}^\omega$ to be a word such that $E(y) = \alpha$,
	where this word is obtained through the construction 
	presented in Remark~\ref{rem:currieConstruction},
	hence it starts
	with $001$ and we can write
	\[
		y = 00\tilde{y},
	\]
	where $\tilde{y}$ begins with $1$.
	Moreover, since we are assuming $\ell(w)\ge 3$,
	we notice also that $0^{\ell(w)}$
	appears only once in $\tilde{w}0^{\ell(w)}\tilde{y}$.
	Indeed, $\tilde{w}$ is either the empty word
	or it ends with $1$ and $\ell(\tilde{w}) < \ell(w)$.
	Additionally, this implies 
	\begin{equation}
		\label{eq:weakTheorem-proof1}
		E(\tilde{w}0^{\ell(w)}) = \ell(w).
	\end{equation}
	Since $\tilde{y}$ starts with $1$
	and it does not contain $000$ as a subword,
	we have
	\begin{equation}
		\label{eq:weakTheorem-proof2}
		E(0^{\ell(w)}\tilde{y}) = \alpha.
	\end{equation}

	We claim
	\[
		E(w0^n\tilde{y}) = E(\tilde{w}0^{\ell(w)}\tilde{y}) = \alpha,
	\]
	thus obtaining the thesis due to the arbitrariness of $\alpha$
	and the fact that $n$ is independent of $\alpha$.
	First, let us notice that if $\tilde{w} = \epsilon$,
	then the claim directly follows from~\eqref{eq:weakTheorem-proof2}.
	Hence, from now on let us assume $\tilde{w}\ne \epsilon$.

	Since $E(00\tilde{y}) = \alpha$, 
	we have that
	$E(\tilde{w}0^{\ell(w)}\tilde{y}) \ge \alpha$;
	therefore, looking for a contradiction,
	let us assume that $E(\tilde{w}0^{\ell(w)}\tilde{y}) > \alpha$.
	This means that there exists a subword $p$ of $\tilde{w}0^{\ell(w)}\tilde{y}$
	such that $p=v^tv'$,
	where $v'$ is a prefix of $v$
	and $t+\frac{\ell(v')}{\ell(v)}>\alpha\ge\ell(w)$.
	Notice that in this case, since $\ell(w) \ge 3$,
	also $t$ is greater than or equal to $3$,
	so $vvv$ is a subword of $p$.

	By~\eqref{eq:weakTheorem-proof1}
	and~\eqref{eq:weakTheorem-proof2},
	we have that the first digit of $p$ is in $\tilde{w}$	and the last 
	is in $\tilde{y}$.
	The latter implies that $v$ should contain the digit $1$.
	Therefore, $v$ cannot be a subword of $0^{\ell(w)}$.
	Moreover, since $t\ge 3$ and $0^{\ell(w)}$
	appears only once in $\tilde{w}0^{\ell(w)}\tilde{y}$,
	we have that $0^{\ell(w)}$ cannot be a subword of $v$.

	Since $0^{\ell(w)}$ is a subword of $p = v^tv'$
	and $v$ contains the digit $1$,
	we deduce that 
	$0^{\ell(w)}$ is a subword of $vv$. Using again the fact that $v$ contains the digit $1$,
	we obtain that 
	$vvv$, which is a subword of $p$,
	contains at least twice $0^{\ell(w)}$,
	which is impossible.
\end{proof}

\section{The critical exponent functions and their properties}
\label{sec:properties}
This section reviews the generic critical exponent functions,
denoted by $\kappa_n$,
and their analytical and dynamical properties
with respect to the error of~\cite[Theorem 7]{Critical}. 
In the following, we present the definition of the $2$-critical exponential function, denoted by $\kappa_2$,  which we consider as model case.

\begin{definition}
	The critical exponent function
	$\kappa_2\colon [0,1]\to [0,1]$ is defined as follows:
	\begin{equation*}
		\kappa_2(x) \coloneqq \inf_{u \in \mathcal{L}(w_x)} \frac{1}{E(u)}.
	\end{equation*}
\end{definition}
\begin{remark}
	Using the convention $1/\infty = 0$,
	we can write 
	$\kappa_2(x) = 1/E(w_x)$.
	In other words, for $x\in[0,1]$, $\kappa_2(x)$ is 
	the inverse of the critical exponent 
	of the binary expansion of $x$.
\end{remark}

More generally,  for each $x \in [0,1]$, we can define the critical 
exponent function $\kappa_n\colon [0,1] \to [0,1]$
at $x$ as the inverse of the critical exponent 
of the expansion of $x$ in base $n$.
For any $n \in \mathbb{N}$, $n \ge 2$,
and $x\in[0,1],$  we denote by $w_{x,n}$
the expansion of $x$ in base $n$, namely
$x = (0.w_{x,n})_n$.
Hence, for every integer $n \ge 2$, the function 
$\kappa_n\colon [0,1]\to [0,1]$ is defined by:
\begin{equation*}
	\kappa_n(x) \coloneqq \inf_{u \in \mathcal{L}(w_{x,n})}\frac{1}{E(u)}.
\end{equation*}

\subsection{The model case $\kappa_2$}

Let us consider the properties of the 
function $\kappa_2$ as a model case for 
all the others functions $\kappa_n$.  More precisely, 
we are going to study the results of~\cite{Critical} by showing how the properties of this function vary with respect to the observations made on~\cite[Theorem 7]{Critical}.

Let us denote the set of the zeros of $\kappa_2$ by $C_2$,
and its complementary set by $D_2$, i.e.,
\[
C_2 := \big\{x \in [0,1]: \kappa_2 = 0\big\},
\qquad D_2 \coloneqq [0,1]\setminus C_2.
\]
Based on the presentation in~\cite{Critical},
we have the following results,
which do not rely on~\cite[Theorem 7]{Critical}:
\begin{itemize}
	\item the function $\kappa_2$ is upper semi-continuous,
		and therefore it is of Baire class $1$
		(cf.~\cite[Proposition 11]{Critical}),
		i.e., it can be obtained as the pointwise limit
		of a sequence of continuous functions;
	\item  the set of points of continuity of the function $\kappa_2$
		coincides with the set $C_2$.
		(cf.~\cite[Proposition 12]{Critical});
	\item the set $C_2$ is co-meagre  
		and it has full Lebesgue measure
		(cf.~\cite[Proposition 14]{Critical});
	\item both $C_2$ and $D_2$ are dense in $[0,1]$,
		and they are neither closed nor open
		(cf.~\cite[Proposition 15 and Proposition 16]{Critical});
\end{itemize}

A result heavily dependent on~\cite[Theorem 7]{Critical}
is~\cite[Proposition 19]{Critical},
which states that if $\kappa_2(x) \ne 0$,
then $[0,\kappa_2(x)[$ is a subset of the 
left or right range of $\kappa_2$ at $x$,
which are the sets of values 
the function $\kappa_2$ assume on every open
left or right neighbourhood of $x$, respectively.
More formally, 
if we denote them by $R^-(\kappa_2,x)$ and $R^+(\kappa_2,x)$,
they are defined as follows:
\begin{align*}
	R^-(\kappa_2,x)\coloneqq \left\{
		\alpha \in \mathbb{R}:\kappa_2^{-1}(\alpha)\cap(x-\delta,x)\ne\emptyset,
		\forall\delta>0
	\right\},\\
	R^+(\kappa_2,x)\coloneqq \left\{
		\alpha \in \mathbb{R}:\kappa_2^{-1}(\alpha)\cap(x,x+\delta)\ne\emptyset,
		\forall\delta>0
	\right\}.
\end{align*}
With this notation,~\cite[Proposition 19]{Critical}
ensured that if $\kappa_2(x) = 0$,
then both $R^-(\kappa_2,x)$ and $R^+(\kappa_2,x)$
contain only $0$.
On the other hand,
if $\kappa_2(x) > 0$
then at least one of the following holds:
\[
	[0,\kappa_2(x)[ \subset R^-(\kappa_2,x)
	\quad\text{or}\quad
	[0,\kappa_2(x)[ \subset R^+(\kappa_2,x).
\]
It remains uncertain if~\cite[Proposition 19]{Critical} holds true or not,
as Conjecture~\ref{conj:ourCon} could be sufficient to prove it.
We remark that the result does not follow directly from the conjecture,
since one has to ensure that for any (finite) prefix $w$ of $w_{x}$
and for any $\alpha \ge \mathcal{E}_{w}$ 
the infinite word $y_{\alpha}$ such that $E(wy_\alpha) = \alpha$
can be chosen so that 
$(0.wy_{\alpha})_2 < x$
(or $(0.wy_{\alpha})_2 > x$) for every $\alpha$.
Otherwise, in principle, some strange phenomena could occur,
such as, for instance, that
$R^-(\kappa_2,x) = [0,\kappa_2(x)] \cap \mathbb{Q}$
while
$R^+(\kappa_2,x) = [0,\kappa_2(x)] \setminus \mathbb{Q}$.

 While the full validity of~\cite[Proposition 19]{Critical}
remains uncertain,
the result holds for the special case 
of $x_\tau$,
which is the number whose binary expansion is the 
Thue-Morse sequence, namely
$x_\tau = (0.\tau)_2$.
Indeed, we have the following result.
\begin{lemma}
	\label{lem:leftRange-ThueMorse}
	The interval $[0,1/2)$ is a subset of $R^-(\kappa_2,x_\tau)$.
\end{lemma}
\begin{proof}
	Recalling that for any $n \in \mathbb{N}$ the 
	$(2^n + 1)$-th digit of $\tau$ is $1$
	(see Remark~\ref{rem:ThueMorse-bitwiseConstruction}),
	for any $\epsilon > 0$
	there exists $m \in \mathbb{N}$
	such that 
	$I_{\tau_{(m)}0} \subset (x_\tau - \epsilon,x_\tau)$,
	where $\tau_{(m)}$ is the prefix of length $2^m$ of 
	the Thue-Morse sequence and $I_w$ is the open interval 
	characterized by $w$.
	Indeed, for any $\epsilon > 0$
	there exists a positive integer $m$
	such that $I_{\tau_{(m)}} \subset (x_\tau - \epsilon,x_\tau + \epsilon)$;
	moreover if 
	$y \in I_{\tau_{(m)}0} \subset I_{\tau_{(m)}}$,
	then 
	\[
		y = (0.\tau_{(m)}0\ldots)_2
		< (0.\tau_{(m)}1\, \tau_{2^m + 2} \tau_{2^m + 3}\dots)_2
		= (0.\tau)_2 = x_\tau,
	\]
	hence $I_{\tau_{(m)}0} \subset (x_\tau-\epsilon,x_\tau)$.

	By Proposition~\ref{prop:w-subword-TM}, it suffices to show 
	that $\tau_{(m)}0$ is a subword of $\tau$.
	Using again the construction of $\tau$
	given in Remark~\ref{rem:ThueMorse-bitwiseConstruction},
	and denoting by $\overline{\tau_{(m)}}$
	the bitwise negation of $\tau_{(m)}$,
	we have that $\tau$ starts as follows:
	\[
	 \tau = \tau_{(m)}\,
	 \overline{\tau_{(m)}}\,
	 \overline{\tau_{(m)}}\,
	 \tau_{(m)}\,
	 \overline{\tau_{(m)}}\,
	 \tau_{(m)}\,
	 \tau_{(m)}\,
	 \overline{\tau_{(m)}}\,
	 \dots
	 \, .
	\]
	As a consequence, $\tau_{(m)}\tau_{(m)}$
	is a subword of $\tau$, ans since $\tau_{(m)}$
	starts with $0$,
	also $\tau_{(m)}0$ is a subword of $\tau$,
	and this ends the proof.
\end{proof}

Lemma~\ref{lem:leftRange-ThueMorse} is sufficient to restore some properties
that in~\cite{Critical} were proved using ~\cite[Proposition 19]{Critical},
such as the existence of almost fixed points and fixed points of $\kappa_2$.
We recall that an almost fixed point of a function is a point that belongs
to the interior of the left or right range of the function at that point.
More formally, we have the following definition.
\begin{definition}
A point $x\in [0,1]$ is an \emph{almost fixed point}
of $\kappa_2$ if and only if
\[
	x\in \mathrm{Int}(R^-(\kappa_2,x))\cup\mathrm{Int}(R^+(\kappa_2,x)).
\]
We denote the set of almost fixed points of $\kappa_2$
as $\mathrm{aFix}(\kappa_2)$.
\end{definition}

In~\cite{Critical}, a characterization of the fixed points was given,
namely $x \in \mathrm{aFix}(\kappa_2)$ if and only if $\kappa_2(x) > x$
(cf.~\cite[Corollary 24]{Critical}).
However, this result relies on~\cite[Proposition 19]{Critical},
hence its validity is no longer ensured.
As a consequence, the uncountability of $\mathrm{aFix}(\kappa_2)$
can't be ensured (cf.~\cite[Corollary 25]{Critical}).
The only thing we can conclude,
by employing the upper semi-continuity of $\kappa_2$,
is that if $x\in [0,1]$ is an almost fixed point of $\kappa_2$,
then $\kappa_2(x) > x$.
However, based on Lemma~\ref{lem:leftRange-ThueMorse},
we can establish that $x_\tau \in \mathrm{aFix}(\kappa_2)$.
Indeed, since $x_\tau < \kappa_2(x_\tau) = 1/2$,
we have that
$x_\tau \in (0,1/2) \subset \mathrm{Int}(R^-(\kappa_2,x))$.
This is also sufficient to show 
the existence of a fixed point of $\kappa_2$.
Indeed,~\cite[Corollary 26]{Critical}
stated that, for any $x_0 \in \mathrm{aFix}(\kappa_2)$
and $\epsilon > 0$, there exists a fixed point of $\kappa_2$
in $(x_0 - \epsilon,x_0 + \epsilon)$.
This is not correct, 
but the same proof of~\cite[Corollary 26]{Critical}
shows the following result.
\begin{corollary}
	\label{cor:fixedPoint}
	For any $\epsilon > 0$, there exists $z \in (x_\tau -\epsilon,x_\tau)$
	such that $\kappa_2(z) = z$.
\end{corollary}
This result, which may appear a bit surprising, ensures the existence of
points that coincide with the inverse 
of the critical exponent of their binary expansion,
and they are (at least) countably many,
allowing~\cite[Corollary 27]{Critical} to still hold true.

\bigskip
The analysis in~\cite{Critical} continued
by considering the topological entropy of $\kappa_2$.
We recall that the topological entropy $h$ of an interval map $f$ can be
defined as follows
(see, e.g., \cite{bowen1971entropy,dinaburg1970correlation}):
\begin{equation}
	h(f)=\lim_{\varepsilon \to 0} \left( \limsup_{n \to \infty}  \frac{1}{n} \log \left| (n,\varepsilon) \right| \right),
	\label{entropy_}
\end{equation}
where $\left| (n,\varepsilon) \right|$ is
the maximum cardinality of an $\varepsilon$-separated set
in the following metric $d_n$:
\begin{equation} \label{N_metric}
	d_n(x_1,x_2) \coloneqq
	\max\big\{\left| f^i(x_1)-f^i(x_2)\right| : 0\le i \le n\big\}.
\end{equation}
In other words, the topological entropy can be interpreted as an estimate of the exponential increase rate of the number of topologically distinguishable orbits
when the resolution power diverges.

Specifically, in~\cite[Proposition 31]{Critical} it was shown that 
$\kappa_2$ has infinite topological entropy
(cf.~\cite[Proposition 31]{Critical}).
This is a direct consequence of the existence
of horseshoes of every order
(cf.~\cite[Proposition 30]{Critical}).
The latter result that can be recovered
by using Proposition~\ref{prop:w-subword-TM},
instead of the (now uncertain)~\cite[Proposition 19]{Critical}.
For a clearer understanding,
we report here the revised proof.

\begin{proposition}[Proposition 30 of \cite{Critical}]
	\label{prop:horse}
	The function $\kappa_2$ has an $m$--horseshoe for any $m \ge 2$,
	i.e.,
	there exists an interval $J\subset [0,1]$
	and a family of pairwise closed disjoint intervals
	$I_1,\dots,I_m \subset J$ such that
	$J \subset \kappa_2(I_k)$,
	for all $k = 1,\dots,m$.
\end{proposition}
\begin{proof}
	For any $k = 1,\dots,m$,
	let $\tau_{(k)}$
	be the prefix of 
	length $2^k$ of the Thue-Morse sequence,
	and let $I_k$ be the closed
	cylinder set characterized by $\tau_{(k+1)}0$,
	i.e.,
	\[
		I_k = [(0.\tau_{(k+1)}0)_2,(0.\tau_{(k+1)}1)_2].
	\]
	Therefore, we have 
	\begin{align*}
		\text{if } x \in I_1,
		\text{then } x = &(0.01 10 0\dots)_2,\\
		\text{if } x \in I_2,
		\text{then } x = &(0.01 10 1001 0\dots)_2,\\
		\text{if } x \in I_3,
		\text{then } x = &(0.01 10 1001 10010110 0\dots)_2,
	\end{align*}
	and so forth, recalling that, if $x$ is a dyadic rational,
	we consider its binary expansion whose digits are ultimately $1$.
	Since the $(2^k + 1)$-th digit of the Thue-Morse sequence is 
	$1$ for any $k\ge 1$,
	the intervals $I_k$ are pairwise disjoint
	and if $x\in I_k$ then $x < x_\tau$.
	For any $k\ge 1$, the word $\tau_{(k+1)}0$
	is a subword of $\tau_{(k+1)}\tau_{(k+1)}$,
	which is a subword of the Thue-Morse sequence,
	as one can infer from the bitwise negation construction 
	of $\tau$ given in Remark~\ref{rem:ThueMorse-bitwiseConstruction}.
	According to Proposition~\ref{prop:w-subword-TM},
	$\kappa_2(I_k) = [0,1/2)$
	for any $k = 1,\dots, m$.
	Therefore, by setting $J = [0,x_\tau]$, we obtain the thesis,
	given the generality of $m$.
\end{proof}

We end the exploration of $\kappa_2$
by recovering some results on its dynamical properties.
First, we can ensure the following result
(cf.~\cite[Lemma 32]{Critical}),
whose original proof relied on~\cite[Theorem 7]{Critical},
but a minor revision of it is sufficient to 
demonstrate the claim.
\begin{lemma}
	\label{lem:near0}
	For any
	$n \ge 1$ and 
	$y \in [0,1/2^n]$,
	there exists $x \in [0,1/2^{2^n}]$
	such that $\kappa_2(x) = y$.
\end{lemma}
\begin{proof}
	If $y=0$,
	then $x=0$.
	If $y\in(0,1/2^n]$, then $1/y\ge 2^n$.
	By Proposition~\ref{prop:weakTheorem},
	we have that $[2^n,+\infty] \subset E(\mathcal{P}(0^{2^{n}}))$.
	Then, there exists $\alpha_y\in\{0,1\}^\omega$
	such that $E(0^{2^{n}}\alpha_y)=1/y$.
	As a consequence, the number $x=(0.0^{2^{n}}\alpha_y)_2$
	is such that $\kappa_2(x)=y$ and $x\in [0,1/2^{2^{n}}]$.
\end{proof}

Lemma~\ref{lem:near0}
and Proposition~\ref{prop:weakTheorem}
lead to the topological mixing property of 
$\bar\kappa_2 \coloneqq \kappa_2|_{[0,1/2]}$.
Indeed, they are sufficient to prove that 
for any open interval $I \subset [0,1]$
there exists $n\in\mathbb{N}$ such that 
$\kappa_2^n(I) = [0,1/2]$
(cf.~\cite[Proposition 33]{Critical}),
from which the topological mixing property
directly follows.

The last result about the properties of $\kappa_2$
was~\cite[Lemma 36]{Critical},
which is about a kind of ``finite-type transitivity''
and ``finite-type Li–Yorke pairs''.
We recall that a point is called \textit{transitive} if its orbit is dense,
and that a \textit{Li-Yorke pair} is a pair of points $(x,y)$ such that
\[
	\limsup_{n \to \infty}\abs{\kappa_2^{n}(x) - \kappa_2^n(y)} > 0
	\quad \text{and} \quad
	\liminf_{n \to \infty}\abs{\kappa_2^{n}(x) - \kappa_2^n(y)} = 0.
\]
In~\cite{Critical}, different steps of the proof of~\cite[Lemma 36]{Critical}
are based on~\cite[Theorem 7]{Critical},
but we are going to show that they can be 
recovered by using Proposition~\ref{prop:w-subword-TM}
and Proposition~\ref{prop:weakTheorem}.

\begin{lemma}[Lemma 36 of \cite{Critical}]
	\label{lem:finite-LiYorke}
	For any $w_{(1)},w_{(2)}\in\{0,1\}^+$
	such that both $(0.w_{(1)})_2$ 
	and $(0.w_{(2)})_2$ are less than $1/2$,
	there exist $x_1,x_2 \in I_{w_{(1)}}$
	and a number $n\in\mathbb{N}$
	such that:
	\begin{itemize}
		\item[(i)]:
			$\kappa_2^{n+2}(x_1),\kappa_2^{n+2}(x_2)\in I_{w_{(2)}}$;
		\item[(ii)]:
			$\kappa_2^{n}(x_2)-\kappa_2^{n}(x_1) > 1/8$.
	\end{itemize}
\end{lemma}
\begin{proof}
	Since $(0.w_{(2)})_2 < 1/2$,
	and both $00$ and $01$ are subwords of the Thue-Morse sequence,
	by Proposition~\ref{prop:w-subword-TM}
	there exist $\alpha_1,\alpha_2 \in \{0,1\}^\omega$
	such that,
	setting $z_1 = (0.00\alpha_1)_2$
	and $z_2 = (0.01\alpha_2)_2$,
	we have 
	$\kappa_2(z_1),\kappa_2(z_2)\in I_{w_{(2)}}$.
	Since $z_1 < 1/4$, by Lemma~\ref{lem:near0}
	there exists $y_1 < 1/16$ such that $\kappa_2(y_1) = z_1$.
	By using again Proposition~\ref{prop:w-subword-TM},
	there exists $y_2 = (0.010\dots)_2$
	such that $\kappa_2(y_2) = z_2$.
	Therefore, we have
	$\kappa_2^2(y_1),\kappa_2^2(y_2) \in I_{w_{(2)}}$
	and $y_2 - y_1 > 1/8$,
	and the proof is complete by showing 
	the existence of a positive integer $n$
	such that $\kappa_2^n(I_{w(1)}) = [0,1/2]$.
	By applying Proposition~\ref{prop:weakTheorem}
	to $w = w_{(1)}$,
	we have
	$\kappa_2(I_{w(1)}) = [0,1/\ell(w_{(1)})]$.
	Then, through an iterative procedure involving 
	Lemma~\ref{lem:near0},
	we can identify an $n \in \mathbb{N}$
	such that $\kappa_2^n(I_{w(1)}) = [0,1/2]$,
	thus concluding our argument.
\end{proof}

Summarizing,
after reviewing all the properties of $\kappa_2$ studied in~\cite{Critical},
we find that all of them hold true
except for the following ones:
\begin{itemize}
	\item The property of being left- or right-Darboux
		at every point, as stated in~\cite[Proposition 19]{Critical},
		remains unclear if this property holds or not;
	\item The characterization of the almost fixed points of $\kappa_2$,
		specifically in~\cite[Corollary 24]{Critical},
		and the subsequent uncountability of them
		(cf.~\cite[Corollary 25]{Critical});
	\item The existence of the fixed point can be granted 
		only on any left-neighbourhood of $x_\tau$,
		not in a neighbourhood of any almost fixed point
		(cf.~\cite[Corollary 26]{Critical}
		and~\ref{cor:fixedPoint}).
\end{itemize}

\subsection{On the functions $\kappa_n$}

In~\cite{Critical},
several properties
of the general function $\kappa_n$,
with $n \ge 3$, were
explored by examining its restriction
to the set
\[
	A_2^n\coloneqq\left\{x \in [0,1]: w_{x,n} \in \{0,1\}^\omega\right\},
\]
where $w_{x,n}$ is the $n$--base expansion of $x$,
and by employing~\cite[Theorem 7]{Critical}.
In particular,~\cite[Corollary 42]{Critical},
which plays an analogous role of~\cite[Proposition 19]{Critical}, 
is no longer granted.
However, as for the case of $\kappa_2$,
most of the properties can be regained by 
using Proposition~\ref{prop:w-subword-TM}
and
Proposition~\ref{prop:weakTheorem}.
In particular, we have that 
$x_{\tau,n} = (0.01101001\dots)_n$
is an almost fixed point for $\kappa_n$
and for any $\epsilon > 0$
there exists a fixed point in the interval $(x_{\tau,n}- \epsilon,x_{\tau,n})$.
Thus, there are (at least) countably many
fixed points for $\kappa_n$.
Moreover, with some minor adjustments on the proof,
it can be showed that 
$\kappa_n$ admits horseshoes of every order
(cf. Proposition~\ref{prop:horse}),
so it has infinite topological entropy,

\subsection{On the function $\kappa$}
Finally, in~\cite{Critical} the properties of the function $\kappa\colon [0,1] \to [0,1]$, which was defined  by considering the supremum of $\kappa_n$
among all $n \ge 2$, were studied.  More precisely, 
\begin{definition}
We indicate by $\kappa\colon [0,1] \to [0,1]$ the function given by
\[
	\kappa(x)\coloneqq \sup\left\{\kappa_n(x): n \in \mathbb{N}, n \ge 2\right\}.
\]
\end{definition}

Since~\cite[Theorem 7]{Critical} and strictly related
results were never employed in the analysis of $\kappa$,
all the properties presented in~\cite{Critical} for $\kappa$ hold true.
The most important ones are: 
\begin{itemize}
	\item $\kappa$ is $0$ a.e., hence Lebesgue-measurable;
	\item the points where $\kappa$ is different from $0$ 
		are uncountably many;
	\item $\kappa$ is discontinuous at any point $x\in [0,1]$;
	\item the function $\kappa$ belongs to the Baire class $2$
		(but not to the Baire class $1$),
		i.e., it can be obtained as the pointwise limit
		of a sequence of functions of Baire class $1$.
\end{itemize}


\end{document}